\documentclass[leqno,12pt]{amsart}

\usepackage{amsfonts,amssymb,amsgen, stmaryrd, amsopn}
\usepackage[leqno]{amsmath} 
\usepackage[T1]{fontenc}
\usepackage{color}
\usepackage{kpfonts}
\usepackage[latin1]{inputenc} 
\usepackage[french,english]{babel}
\usepackage{graphicx}
\usepackage{graphpap}
\usepackage{hyperref}

\usepackage{xcolor}

\newtheorem{theoreme}{Theorem}

\newtheorem{proposition}{Proposition}[section]
 \newtheorem{lemme}{Lemma}[section]
 \newtheorem{definition}{Definition}[section]
 \newtheorem{rem}{Remark}[section]
 
\numberwithin{equation}{section}

\renewcommand\Re{\mathrm{Re}\,} 
\newcommand\R{{\mathbb R}} \newcommand\N{{\mathbb N}}
\newcommand\Z{{\mathbb Z}} 
\newcommand\C{{\mathbb C}}

\newcommand\supp{{\mathrm{supp}\hskip 0.05cm}}

\renewcommand{\Re}{  \text{Re}   }

 \def\cdotv{\raise 2pt\hbox{,}}
\def\div{\hbox{div}\,}

\renewcommand\Re{\mathrm{Re}\,}

\makeatletter
\def\@tvsp{\mathchoice{{}\mkern-4.5mu}{{}\mkern-4.5mu}{{}\mkern-2.5mu}{}}
\def\ltrivert{\left|\@tvsp\left|\@tvsp\left|}
\def\rtrivert{\right|\@tvsp\right|\@tvsp\right|}
\makeatother

 \def\cdotv{\raise 2pt\hbox{,}}
\def\div{\hbox{div}\,}

\begin{document}
\title[Standing waves for NLS with partial confinement]{Existence and Stability 
of standing waves for supercritical NLS with a Partial Confinement}

\author[Jacopo Bellazzini]{Jacopo Bellazzini}
\address{Jacopo Bellazzini
\newline\indent
Universit\` a di Sassari
\newline\indent
Via Piandanna 4, 07100 Sassari, Italy}
\email{jbellazzini@uniss.it}

\author[Nabile Boussa\"id]{Nabile Boussa\"id}
\address{Nabile Boussa\"id
\newline\indent
Laboratoire de Math\'ematiques (UMR 6623)
\newline\indent
Univ. Bourgogne Franche-Comt\'e
\newline\indent
16, Route de Gray 25030 Besan\c{c}on Cedex, France}
\email{nabile.boussaid@univ-fcomte.fr}

\author[Louis Jeanjean]{Louis Jeanjean}
\thanks{This work has been carried out in the framework of the project NONLOCAL (ANR-14-CE25-0013),
funded by the French National Research Agency (ANR). J.B. was supported by Gnampa project 2016 "Equazioni nonlineari dispersive". N.V. was supported by the research project PRA 2016.}
\address{Louis Jeanjean
\newline\indent
Laboratoire de Math\'ematiques (UMR 6623)
\newline\indent
Univ. Bourgogne Franche-Comt\'e
\newline\indent
16, Route de Gray 25030 Besan\c{c}on Cedex, France}
\email{louis.jeanjean@univ-fcomte.fr}

\author[Nicola Visciglia]{Nicola Visciglia}
\address{Nicola Visciglia
\newline\indent Dipartimento di Matematica 
\newline\indent
Università Degli Studi di Pisa
\newline\indent
Largo Bruno Pontecorvo 5 I - 56127 Pisa
}
\email{viscigli@dm.unipi.it}

\date{}
 
 \begin{abstract}
We prove the existence of orbitally stable ground states to NLS with a partial confinement together with qualitative and symmetry properties. This result is obtained for nonlinearities which are $L^2$-supercritical, in particular we cover the physically relevant cubic case. The equation that we consider is the limit case of the cigar-shaped model in BEC. 
\end{abstract}

 \maketitle
 \par \noindent

\section{Introduction}

The aim of this work is the study of the existence, stability, qualitative and symmetry properties 
of standing waves associated with the following Cauchy problem:
\begin{equation}\label{NLS}
\begin{cases}i\partial_t u +\Delta u -(x_1^2+ x_2^2) u = - u|u|^{p-1}, \quad
(t, x_1, x_2, x_3)\in \R \times \R^3, 
\\
u(0, x)=\varphi,
\end{cases}
\end{equation}
with $1+4/3\leq p<5$. Notice that the  range of nonlinearities are  respectively $L^2$-supercritical and $H^1$-subcritical. It is well known that in this range the ground states of the translation invariant Nonlinear Schr\"odinger Equation (NLS), i.e. \eqref{NLS} where
we remove the term $(x_1^2+ x_2^2) u$, are unstable by blow-up, see~\cite{BC,Ca}. As we shall see the situation changes completely if we add a confinement.

Note that the
physically relevant cubic nonlinearity $p=3$ is covered by our assumptions. Cubic NLS,  often referred as  Gross-Pitaevskii equation (GPE), is  very important in physics. NLS when $p=3$  with an external trapping potential  gives a good description for the   Bose-Einstein condensate (BEC).
  Bose-Einstein condensate  consists in a macroscopic ensemble of bosons that at  very low temperature  occupy  the same quantum states. Such kind of condensate  has been experimentally observed only in the last two decades \cite{AE} and this fact has stimulated a wave of
  activity on both the theoretical and the numerical side. In the experiment BEC is observed in presence of a confined potential trap and its macroscopic  behavior  strongly  depends on the shape of this trap potential. From the mathematical side, the properties of BEC at temperature much smaller than the critical temperature can be well described by 
the three-dimensional Gross-Pitaevskii equation (GPE) for the macroscopic wave function $u$ given by, see. e.g. \cite{GSS}, \cite{DG}
\begin{equation}
\label{eq:evolution}
i \hbar \partial_t u + \frac{\hbar^2}{2m}\Delta u- W(x_1,x_2,x_3) u=N U_0|u|^2 u, \quad (t,x_1,x_2,x_3) \in \R \times \R^3
\end{equation}
where $\hbar$ is the reduced Plank constant, $m$ the particle mass, $N$ the number of particles, $U_0$ is a constant accounting for the interaction among the particles and $W$ an external potential.

A rigorous derivation of GPE  from the $N$ body quantum system of particles has been deeply investigated
in recent years  \cite{ESY,ESY2,LS,LS2} as well as numerical analysis associated to the dynamics of  GPE, see e.g \cite{BJM}. 
GPE incorporates information about the trapping potential as well as the interaction among the particles. Repulsive interactions correspond to the case $U_0>0$, while
attractive interactions to the case $U_0<0$. The trapping potential in experiments is usually harmonic, i.e. $W=\omega_{x_1}^2x_1^2+\omega_{x_2}^2x_2^2+\omega_{x_3}^2x_3^2$ with $\omega_{x_1}\neq 0,\omega_{x_2} \neq 0,\omega_{x_3} \neq 0$.
In this case it is known, see~\cite{FuOh1}, that there exist orbitally stable solitary waves. We show that in the intermediate case of partial confinement stable solutions still exist. 
Our assumptions cover the case of Bose-Einstein condensate with attractive interactions and partial confinement. The latter includes the limit case of the so-called cigar-shaped model, see~\cite{BaCa}.

To our knowledge existence and stability of standing states for BEC in presence of a partial confinement has not been studied in the literature. 

\medskip

In order to go further and to present our main results,
we fix some definition that will be useful in the sequel. We consider
$$H=\big \{u\in H^1(\R^3;\C)\mbox{ s.t. } \int_{\R^3} (x_1^2+x_2^2) |u(x_1, x_2, x_3)|^2 dx < \infty
\big \}$$
where $x=(x_1, x_2, x_3)\in \R^3$.
We also introduce 
$$E(u)=\frac 12 \|u\|_{\dot H}^2 - \frac 1{p+1}\int_{\R^3} |u(x_1, x_2, x_3)|^{p+1} dx,$$
where $$\|u\|_{\dot H}^2 = \sum_{i=1}^3 \int_{\R^3} |\partial_{x_i} u(x_1, x_2, x_3)|^2 dx + 
\int_{\R^3} (x_1^2+x_2^2) |u(x_1, x_2, x_3)|^2 dx.$$
We shall need the following sets:
\begin{equation}\label{def:SB}
S_r=\big \{ u\in H \mbox{ s.t. }\int_{\R^3} |u(x_1, x_2, x_3)|^2 dx=r^2\big \},\quad B_R=\{u\in H\mbox{ s.t. }\, \|u\|_{\dot H}^2\leq R^2\}.
\end{equation}
Following \cite{CL} a first idea to construct orbitally stable standing waves
to \eqref{NLS}
is to consider the following constrained minimization problems
$$I_r=\inf_{u\in S_r} E(u)$$
and to try to prove the compactness of the minimizing sequences 
up to the action of the group of the symmetries.
If one assumes that $1<p<1+4/3$ then $I_r>-\infty$ for any $r>0$ and this approach by global minimization would indeed work. On the contrary in the case
$p>1+ 4/3$, we have 
$I_r=-\infty$ for any $r>0$. Indeed notice that
if we fix $\psi\in C^\infty_0(\R^3)$ and we define
$\psi_\lambda(x)=\lambda^{3/2} \psi(\lambda x)$
then
$$\|\psi_\lambda(x)\|_{L^2}=\|\psi(x)\|_{L^2}
\hbox{ and }
E(\psi_\lambda)=\lambda^2 \int_{\R^3} |\nabla_x \psi|^2 dx - \lambda^{\frac{3(p-1)}2}
\|\psi\|_{L^{p+1}}^{p+1}+ \frac 1{\lambda^2} \int_{\R^3} (x_1^2+x_2^2) 
|\psi|^2 dx$$
and hence 
$E(\psi_\lambda)\rightarrow -\infty$
as $\lambda\rightarrow \infty$.

For this reason and since we are assuming that $1+4/3\leq p<5$ 
we shall construct orbitally stable solutions by considering a suitable
localized version of the minimization problems
above.
More precisely for every given $\chi>0$
we consider the following localized minimization problems:
\begin{equation}\label{Jrchi}J_r^\chi=\inf_{u\in S_r\cap B_\chi} E(u),\end{equation}
recall definitions~\eqref{def:SB}. Clearly $J_r^\chi>-\infty$ but it is worth noticing that in principle it could
be $S_r\cap B_\chi=\emptyset$, as a consequence
of the following computation:
\begin{equation*} 2 \int_{\R^3} |u|^2 dx=\int_{\R^3} |u|^2 \div (x_1, x_2) dx = -\int_{\R^3} (\partial_{x_1} (|u|^2) x_1 +
\partial_{x_2} (|u|^2) x_2) dx$$$$
\leq  2 \|\nabla_{x_1, x_2} u\|_{L^2}
\sqrt{\int_{\R^3} (x_1^2+ x_2^2) |u|^2 dx}
\leq \|u\|_{\dot H}^2.
\end{equation*}

The first aim of this work is to show that for every $\chi>0$ 
there exists $r_0=r_0(\chi)>0$ such that
$S_r\cap B_\chi\neq \emptyset$ for $r<r_0$ and moreover
all minimizing sequences to $J_r^\chi$ are compact, up to the action
of translations w.r.t. $x_3$, provided that $r<r_0$.
In order to guarantee that the minimizers of \eqref{Jrchi} are critical points
of $E$ restricted on $S_r$ it is also necessary to show
that they do not belong to the boundary of $B_\chi\cap S_r$. Then it is classical, see for example \cite[Proposition 14.3]{Ka}, that for any minimizer $u$ there exists $\lambda = \lambda(u) \in \R$ such that the Euler-Lagrange equation
$$- \Delta u + (x_1^2+x_2^2)u - u |u|^{p-1} = \lambda u$$
holds. The associated standing wave is then given by $e^{-i \lambda t}u(x)$.
Concerning the stability note that the Cauchy wellposedness in $H$ is established in~\cite{ACD}. For the sake of completeness, we recall the notion of stability of a set ${\mathcal M}\subset H$ under the flow associated with \eqref{NLS} namely
\[
 \forall \varepsilon >0,\, \exists \delta>0, \quad \inf_{v \in {\mathcal M} } \|\varphi - v\|_{H}
 <\delta\Longrightarrow \sup_{t \in \R} \inf_{v \in {\mathcal M}} \|u(t) - v\|_H  <\varepsilon,
\]
where the norm $||\cdot||_H$ is given by
$$\|v\|_H ^2= \|v\|_{\dot H}^2 + \int_{\R^3} |v(x_1, x_2, x_3)|^2 dx.$$
We can now state our first result.
\begin{theoreme}\label{th:main}
Let $1+4/3\leq p<5$ be fixed. For every $\chi>0$ there exists $r_0=r_0(\chi)>0$ such that:
\begin{enumerate}
\item\label{th:mainPoint1} $S_r\cap B_\chi\neq \emptyset, \quad \forall r<r_0;$
\item\label{th:mainPoint2} we have the following inclusions:
$$\emptyset\neq {\mathcal M}_r^{\chi}\subset B_{\chi r}, \quad \forall r<r_0$$
 where
 $${\mathcal M}_r^\chi = \{u\in S_r \cap {B}_{\chi}
\mbox{ s.t. } E(u)=J_r^\chi\};$$
\item\label{th:mainPoint3} the set ${\mathcal M}_r^\chi$ is stable
under the flow associated with \eqref{NLS} for any $r<r_0$.
 \end{enumerate}
\end{theoreme}
\begin{rem} The core of the proof of the stability will follow from the classical argument by~\cite{CL} once we obtain the following: for all $r<r_0$, for all minimizing sequence $(u_n) \subset S_r\cap B_\chi$ 
such that  $\lim_{n\rightarrow \infty} E(u_n)=J_r^\chi$ there exists $(k_n) \subset \R$ such that  
$(u_n(x_1, x_2, x_3 - k_n))$ is compact in $H$. Notice that once this is established
then in particular it implies the existence of minimizers.
\end{rem}

\begin{rem}
By the previous statement (more precisely by the second property)
we can deduce that 
for any couple $\chi_1, \chi_2>0$ we have
$${\mathcal M}_r^{\chi_1} = {\mathcal M}_r^{\chi_2}, \quad \forall r<<1.$$
This shows that the above minimizers are independent of choice of $\chi$ for small $r$. 
\end{rem}

\begin{rem}
The main difficulty in Theorem \ref{th:main} is the lack of compactness,
due to the translation invariance w.r.t. $x_3$. Indeed if one assumes that the potential is fully harmonic, namely that $x_1^2 + x_2^2$ is replaced by $x_1^2 + x_2^2+ x_3^2$ then one benefits from the compactness of the inclusion of $H$ (with $x_1^2 + x_2^2+ x_3^2$) into $L^q(\R^3)$ where $q \in [2,6)$ see for example \cite[Lemma 3.1]{Zh} and the proof of Theorem \ref{th:main} would be rather simple. We overcome
this lack of compactness by using a concentration-compactness argument,
that we have to adapt in a suitable localized minimization problem (indeed global minimizers cannot exist
because of the $L^2$ supercritical character of the nonlinearity).
We believe that this situation is quite new in the literature.
\end{rem}

\begin{rem}
In Lemma \ref{lem:localmin} we prove that
$$
\inf_{S_r \cap B_{\chi r/2}} E(u) <\inf_{S_r \cap (B_{\chi}\setminus B_{\chi r}) }E(u),
\quad \forall r<r_0.
$$
where $r_0 = r_0(\chi) >0$ is given in Theorem \ref{th:main}. Then $E$ has a geometry of local minima. It is important to observe that this geometry still holds when $(x_1^2 + x_2^2)$ in 
\eqref{NLS} is replaced by a general potential $V(x)$. In particular it is the case when $V \equiv 0$ (no potential) or $V(x) = x_1^2 + x_2^2 + x_3^2$ (full harmonic potential). Actually this geometry only depends on the assumption that $p \geq 1 + 4/3$. Now the fact that the geometry implies the existence of a local minimizer depends on a balance between the potential and the strength of the nonlinear term. If $V \equiv 0$ then it is well known that, for $1 + 4/3 < p <5$, no local minimizer exists (in practise any minimizing sequence is vanishing). On the contrary when $V(x) = x_1^2 + x_2^2 + x_3^2$, as already indicated, a local minimizer exists for any $1 + 4/3 \leq p < 5$. In between, namely when only a partial confinement is present, a restriction on $p$ is necessary. This is not apparent in the statement of Theorem \ref{th:main} but it is specified in Remark \ref{rem:extension}. Actually, what seems to be an essential requirement, is that the power $p$ is subcritical with respect to the set of variables bearing no potential. This leads to the restriction $p < 1 + 4/(n-d)$ in Remark \ref{rem:extension} and at the level of the proof this condition is used to establish the equivalent of Lemma \ref{lem:comparison}. Roughly speaking the larger is the number of dimension of confinement, the larger can be the interval of admissible $p \in \R$.\\
The nonlinearity is thus supercritical as the energy functional is not bounded from below. It is supercritical in the dynamical sense as well. In \cite{Carles2002}, for full confinement, the power $p=1+4/3$ is critical as the Gagliardo-Nirenberg level is a threshold for global wellposedness. Even though the equation has no scaling invariance, it can be scaled asymptotically to bear the same criticality levels as the usual nonlinear Schrödinger equation.   
\end{rem}

\begin{rem}
Our results should be compared to the corresponding ones for NLS posed on $\R\times M^2$,
where $M^2$ is a $2d$ compact manifold, see~\cite{TTV}.
In the latter case, there exists a special family of solutions that depend
only on the dispersive variable $\R$ and that can be extended in a trivial way on the transverse direction $M^2$. 

Notice that in our case the situation is different, even if the partial confinement provides some compactness analogously to the $\R\times M^2$ setting. The main point is that 
in our situation we have infinite volume and then the one line solitons cannot be extended in a trivial way on the whole space, by preserving the property of being finite energy. In particular is it is not clear that standing waves solutions exist neither that the approach of Weinstein (see \cite{W}, \cite{RT}), based on a linearization argument, can work to establish the stability.
\end{rem}


Our second result provides properties of the minimizers obtained in Theorem \ref{th:main}.

\begin{theoreme}\label{complex} Every minimizer obtained in  Theorem \ref{th:main} 
(that is in principle $\C$-valued) is of the form $e^{i\theta} f(x_1,x_2,x_3)$ where $f$ is a positive real valued minimizer and $\theta\in \R$. Moreover for some $k \in \R$, $f(x_1,x_2,x_3-k)$ is radially symmetric and decreasing w.r.t. $(x_1,x_2)$ 
and w.r.t. $x_3$. In addition for every fixed $\chi>0$,
for every $r<r_0(\chi)$ and for every $u\in {\mathcal M}_r^{\chi}$
there exists $\lambda=\lambda(u)>0$ such that
$$-\Delta u + (x_1^2+ x_2^2) u -|u|^{p-1} u=\lambda u, \quad \hbox{ on } \R^3$$
with the estimates:
\begin{equation}\label{controlLagrange}
(1 - C r^{p-1}) \Lambda_0 \leq \lambda < \Lambda_0
\end{equation}
where $C>0$ is an universal constant and 
$\Lambda_0=\inf (\hbox{ \em spec }(-\sum_{i=1}^3 \partial_{x_i}^2 + (x_1^2+x_2^2))).$\\
Moreover, we have
\begin{equation}\label{eq:profile}
 \sup_{u\in \mathcal{M}_r^\chi}\left\|u(x_1,x_2,x_3)  -
 \varphi_0(x_3)
\Psi_0(x_1,x_2)\right\|_{\dot{H}} = o(r)
\end{equation}
where $\Psi_0(x_1,x_2)$ is the unique normalized positive eigenvector of the quantum harmonic oscillator $-\sum_{i=1}^2 \partial_{x_i}^2 +x_1^2+x_2^2$
and  $$\varphi_0(x_3)=
\left(\int_{\R^2} u(x_1,x_2,x_3) \Psi_0(x_1,x_2) dx_1dx_2\right).$$
\end{theoreme}

\begin{rem}
The existence of stable standing waves of a $L^2$ supercritical nonlinarity had already been observed in \cite{Fu,FuOh1} and more recently in \cite{BeJe,FuFoKi, NoTaVe}, under compactness assumptions corresponding to complete confinement. In particular in \cite[Theorem 2]{FuOh1}, see also \cite{Fu} for an earlier partial result,
the authors consider the ground state solutions for the equation
\begin{equation}\label{f1}
- \Delta u + V(x) u - |u|^{p-1}u = \lambda u, \quad u \in H^1(\R^3)
\end{equation}
and they show that when $1 <p < 5$ the ground states are orbitally stable 
provided that $\lambda \in (\lambda^*, \lambda_1)$, where $\lambda_1$ corresponds to the bottom of the spectrum of the operator $- \Delta + V(x)$ (which in \cite{FuOh1} is an eigenvalue). Likely the solutions of Theorem \ref{th:main}, which we recall satisfy \eqref{controlLagrange}, correspond to the ground states of 
(\ref{f1}). Note also that in \cite{FuOh2} it is proved on the contrary that the ground states of (\ref{f1}), when $\lambda$ is sufficiently close to $- \infty$, are unstable. In light of the forthcoming Remark \ref{rem:open} we do believe these solutions could be characterized as mountain pass critical points on $S_r$, for $r>0$ small.
\end{rem}

\begin{rem}
The fact that complex valued minimizers are
real valued modulo the multiplication by a complex number of modulus 1,
is well--known in the literature (see \cite{Ca,HaSt2}). 
However the proof that we give here is, we believe, shorter and less involved than the previous ones. 
We also underline that the moving plane technique should provide the symmetry result,
nevertheless we propose an alternative argument based on symmetrization
that as far as we can see has an independent interest, for instance in the context of systems.
\end{rem}

\begin{rem}
Notice that in view of the estimate on the Lagrange multipliers $\lambda$ given in Theorem
\ref{complex},
we can, somehow, consider our results as enlighting a bifurcation phenomena from $\Lambda_0$.
Notice however that $\Lambda_0$ is not an eigenvalue for $-\sum_{i=1}^3\partial_{x_i}^2 + (x_1^2+x_2^2)$, and thus it is 
unclear how to get an existence result via a standard bifurcation argument.
Of course the situation is completely different in the case of a complete 
confinement where we have a compact resolvent
and hence a discrete spectrum.
\end{rem}

In addition to the properties given in Theorem \ref{complex} we can show that our solutions are ground state in the following sense.

\begin{definition}
Let $r>0$ be arbitrary, we say that $u \in S_r$ is a ground state if 
$$E'|_{S_r}(u) =0 \quad \mbox{  and  }\quad
E(u) = \inf \{E(v) \ s.t. \ v \in S_r, E'|_{S_r}(v) =0\}.$$
\end{definition}

We have
\begin{theoreme}\label{ground-state}
Let $1+\frac 43<p<5$. Then for any fixed $\chi>0$ and sufficiently small $r>0$, the local minimizers $u\in \mathcal{M}_r^{\chi}$ are ground states.
\end{theoreme}

\begin{rem}\label{rem:extension}
It is worth noticing that our main results can be easily generalized. 
First of all instead of the confinement $(x_1^2+x_2^2)$, one can consider any potential
$V\big (\sqrt{x_1^2+ x_2^2}\big )$ where $V:(0, \infty)\rightarrow \R$ 
is stricly increasing and unbounded.
Moreover by a direct adaptation of our proofs we can extend our results as to cover the general situation where 
$$i\partial_t u +\Delta u -(x_1^2+...+x_d^2) u = u|u|^{p-1}, \quad
(t, x_1,...,x_d, x_{d+1},...,x_{n})\in \R \times \R^{d}\times \R^{n-d}.$$
Actually assuming that
$$ 1+4/n<p<\min\{1+4/(n-d),1+4/(n-2)\},$$
one can construct orbitally stable standing waves.
In addition these standing waves enjoy the symmetry 
$$w(x_1,..., x_n)=\tilde{w}(x_1,...,x_d, x_{d+1} -k_{d+1},....,x_n-k_n),$$
for some $(k_{d+1},....,k_n) \in \R^{n-d}$ and $\tilde{w}$ is radially symmetry and decreasing  w.r.t. to the first $d$ variables and the least $n-d$ variables separately.
\end{rem}

\begin{rem}\label{rem:open}
We end this introduction by mentioning an open problem related to our results. The solutions found in Theorem \ref{th:main} appear as local minimizers of $E$ on $S_r$ and we recall that $E$ is unbounded from below on $S_r$. This indicates that $E$ has a so-called mountain-pass geometry on $S_r$. Namely it holds that
$$\gamma(c) := \inf_{g \in \Gamma_c}\max_{t \in [0,1]}E(g(t)) >  J_r^{\chi}$$
where
$$\Gamma_c = \big \{g \in C([0,1],S_r) \mbox{ s.t. } g(0) \in {\mathcal M}_r^\chi, g(1) \in S_r \backslash B_{\chi}, E(g(1)) < J_r^{\chi}\big \}.$$
An interesting, but we suspect difficult, question would be to prove the existence of a critical point for $E$ at the level $\gamma(c)$.  In that direction we refer to \cite{BeJe,BeJeLu,Je} where a constrained critical point is obtained by a minimax procedure set on a constraint. If this is true this would establish the existence of at least two critical points on $S_r$ for $r>0$ small. 
\end{rem}

{\bf Notation} In the sequel we shall use, without any further comment, the following notations
as well as the ones introduced along the presentation above:
$$\int f dx=\int_{\R^3} f(x_1, x_2, x_3) dx_1dx_2dx_3,
\int g dx_1 dx_2=\int_{\R^2} g(x_1, x_2) dx_1dx_2,$$$$
\int h dx_3=\int_{\R} h(x_3) dx_3.$$
We shall also denote by $\nabla_x$ the full gradient w.r.t. $x_1, x_2, x_3$,
by $\nabla_{x_1, x_2}$ the partial gradient w.r.t. $x_1, x_2$.\\

{\bf Acknowlegments} The authors thanks R. Carles, A. Farina, J. Van Schaftingen, C.A. Stuart for very useful discussions and suggestions. In particular we thank C.A. Stuart \cite{St} for providing us with elements to simplify the original proof of Theorem \ref{nab}.
We are also indebted to the referees whose comments helped to significantly improve our manuscript.

\section{Spectral theory}

For the sake of completeness, we provide a proof of the lemma below even though it is a classical statement.  Moreover, the proof will be useful for the last assertion of Theorem~\ref{complex}. The aim is to compare the quantities
associated with two spectral problems defined respectively in $3d$ and in $2d$:
$$\Lambda_0=\inf_{\int |w|^2 dx=1} 
\big( \int |\nabla_x w|^2 dx + \int (x_1^2+ x_2^2) |w|^2 dx\big)
$$
and
$$\lambda_0=\inf_{\int |\psi|^2 dx_1 dx_2=1} 
\big (\int |\nabla_{x_1,x_2} \psi|^2 dx_1 dx_2 + \int (x_1^2 + x_2^2) 
|\psi|^2 dx_1dx_2\big ).$$
Note that the spectrum of the $1$-dimensional quantum harmonic oscillator $-\partial_{x_1}^2+x_1^2$ is given by the odd integers (all simple) and the corresponding eigenspaces are generated by Hermite functions. In the $2$-dimensional case, this provides that $\lambda_0=2$, it is simple and a corresponding minimizer is given by the gaussian function $e^{-(x_1^2+x_2^2)}$.
\begin{lemme}\label{lem:spect}
We have the following equality:
$$\lambda_0=\Lambda_0.$$
\end{lemme}

\begin{proof}
We introduce $\Psi_j(x_1, x_2)$ and $\lambda_j$ for $j\geq 0$ such that
$$-\Delta_{x_1, x_2} \Psi_j + (x_1^2 + x_2^2)\Psi_j=\lambda_j \Psi_j, \int |\Psi_j |^2 dx_1 dx_2=1,$$$$ \lambda_j\leq \lambda_{j+1}, j=0,1,2,..... $$
It is well-known that $(\Psi_j)$ is a Hilbert basis for $L^2(\R^2)$. Using Fourier decomposition w.r.t. this basis in the first two variables 
$(x_1, x_2)$ 
we get 
$$w(x)=\sum_{j\geq 1} \Psi_j(x_1, x_2) \varphi_j(x_3).$$
Notice also that if $\|w\|_2=1$ then
$$1=\sum_{j\geq 0} 
(\int |\Psi_j|^2 dx_1 dx_2)(\int |\varphi_j|^2 dx_3)=
\sum_{j\geq 0} \int |\varphi_j|^2 dx_3.$$
Moreover we have
$$ \int |\nabla_x w|^2 dx + \int (x_1^2+ x_2^2) |w|^2 dx
= \int (-\Delta_{x_1, x_2} w + (x_1^2+ x_2^2) w) \bar w dx + \int |\partial_{x_3} w|^2 dx$$
$$=
\sum_{j\geq 0} \lambda_j \int |\varphi_j|^2 dx_3
+ \sum_{j\geq 0} \int |\partial_{x_3} \varphi_j |^2 dx_3
\geq \lambda_0$$
and hence
\[
 \Lambda_0\geq \lambda_0.
\]
To prove the other inequality, we 
choose $$w_n(x)=\Psi_0(x_1,x_2)\varphi_n(x_3), 
\int |\varphi_n|^2 dx_3 =1$$
and notice that arguing as above we get
$$ \int |\nabla_x w_n|^2 dx + \int (x_1^2+ x_2^2) |w_n|^2 dx=
\lambda_0 \int |\varphi_n|^2 dx_3 + \int |\partial_{x_3} \varphi_n |^2 dx_3$$
$$=\lambda_0 + \int |\partial_{x_3} \varphi_n |^2 dx_3.
$$
We conclude by choosing properly $\varphi_n$ such that
$\lim_{n\rightarrow \infty}\int |\partial_{x_3} \varphi_n |^2 dx_3=0$. Notice that this choice can be done
since 
\[ 
\inf_{\int |\varphi|^2 dx_3=1} \int |\partial_{x_3} \varphi|^2 dx_3 =0.\qedhere
\]
\end{proof}

\section{Proof of Theorem \ref{th:main}}

The proof is split in several steps. The main point is to obtain the following compactness. For all $r<r_0$, for any minimizing sequence $(u_n) \subset S_r\cap B_\chi$ 
such that  $\lim_{n\rightarrow \infty} E(u_n)=J_r^\chi$ there exists $(k_n) \subset \R$ such that  
$(u_n(x_1, x_2, x_3 - k_n))$ is compact in $H$. 

Notice as well that Lemma~\ref{lem:localmin} below implies the second point of Theorem~\ref{th:main}.

\subsection{Local Minima Structure for \texorpdfstring{$E(u)$}{E(u)} on \texorpdfstring{$S_r$}{S_r}}

\begin{lemme}\label{lem:localmin} Let $1+4/3\leq p<5$ be fixed, then 
for every $\chi>0$ there exists $r_0=r_0(\chi)>0$ such that:
\begin{equation}\label{eq:1}
S_r\cap B_{\chi }\neq \emptyset, \quad \forall r<r_0;
\end{equation}
\begin{equation}\label{eq:2}
\inf_{S_r \cap B_{\chi r/2}} E(u) <\inf_{S_r \cap (B_{\chi}\setminus B_{\chi r}) }E(u),
\quad \forall r<r_0.
\end{equation}
\end{lemme}

\begin{proof}Let $u_0 \in H$ with $||u_0||_{\dot H} = \chi$ and let $r_0:=||u_0||_{L^2}$. Then the fact that $S_r \cap B_{\chi} \neq \emptyset$ for any $0 < r \leq r_0$ follows by considering $u_r = \frac{r}{r_0}u_0.$
To prove \eqref{eq:2} we borrow some arguments from \cite{Je}.
Notice that we have the following Gagliardo-Sobolev inequality
\begin{equation}\label{GN}
\|u\|_{L^{p+1}}^{p+1} \leq C\|u\|_{L^2}^{\frac{5-p}{2}} \|u\|_{\dot H}^{\frac{3p-3}{2}},
\end{equation}
(see for example \cite{Ca}) and hence
$$\begin{cases}E(u)\geq \frac 12 \|u\|_{\dot H}^2  - C r^{\frac{5-p}{2}}
\|u\|_{\dot H}^{\frac{3p-3}{2}}=f_r(\|u\|_{\dot H}), \quad \forall u\in S_r
\\
E(u)\leq \frac 12 \|u\|_{\dot H}^2 =g_r(\|u\|_{\dot H})), \quad \forall u\in S_r
\end{cases}$$
where
$$\begin{cases}f_r(s)=\frac 12 s^2 - C r^\epsilon s^{2+\delta}\\ 
g_r(s)= \frac 12 s^2 \end{cases}$$
and $$\epsilon=\epsilon(p)=\frac{5-p}{2}>0 , \quad \delta=\delta(p)=\frac{3p-7}{2}\geq 0.$$
Notice that it is sufficient to prove
the existence of $r_0=r_0(\chi)>0$ such that
$$g_r(\chi r/2) < \inf_{s\in (\chi r, \chi)} f_r(s), \quad \forall r<r_0(\chi).$$
In fact this inequality implies:
$$\inf_{S_r\cap B_{\chi r/2}} E(u)< \inf_{S_r\cap (B_{\chi}\setminus
B_{\chi r}) } E(u).$$
Notice that $f_r(s)=\frac 12 s^2 ( 1- C r^\epsilon s^{\delta})>\frac 38 s^2 $
for $s\in (0,\chi)$ and 
for $r<r_0(\chi)<<1$, and hence
$$\inf_{s\in (\chi r, \chi)} f_r(s)\geq \frac 38 \chi^2 r^2.$$
We conclude since
\[
g_r(\chi r/2)= \frac 18 \chi^2 r^2. \qedhere
\]
\end{proof}

\subsection{No Vanishing for Minimizing Sequences}

\begin{lemme}\label{lem:comparison}
Let $1 < p<5$ be fixed, then for every $\chi>0$ there exists 
$r_0=r_0(\chi)>0$ such that
$$\inf_{S_r \cap B_{\chi}} E(u)<r^2\frac{\Lambda_0}2, \quad \forall r<r_0.$$
\end{lemme}

\begin{proof}
First note that $\Lambda_0=\lambda_0$
by Lemma \ref{lem:spect}.
Now we set
$$w(x)=\Psi_0(x_1, x_2) \varphi(x_3), \quad \int |\varphi|^2 dx_3=r^2$$ 
with $\varphi(x_3)$ to be chosen later.
Notice that
$$E(w)= \frac 12 \int |\partial_{x_3} \varphi|^2 dx_3
+\frac 12 \lambda_0 \int |\varphi|^2 dx_3 - \frac 1{p+1}
(\int |\Psi_0|^{p+1} dx_1 dx_2) (\int |\varphi|^{p+1} dx_3 )$$
$$=\frac 12 \int |\partial_{x_3} \varphi|^2 dx_3
+ \frac{\Lambda_0}2 r^2 - \frac 1{p+1}
(\int |\Psi_0|^{p+1} dx_1 dx_2) (\int |\varphi|^{p+1} dx_3 ).$$
In order to conclude it is sufficient
to choose $\varphi$ such that:
\begin{itemize}
\item $\frac 12 \int |\partial_{x_3} \varphi|^2 dx_3
- \frac 1{p+1}
(\int |\Psi_0|^{p+1} dx_1 dx_2) (\int |\varphi|^{p+1} dx_3 )<0.$
\item $\|\Psi_0(x_1, x_2)\varphi(x_3)\|_{\dot H}^2\leq \chi^2$.
\end{itemize}
In order to get the existence of $\varphi(x_3)$ we 
fix $\psi(x_3)$ such that $\int |\psi(x_3)|^2 dx_3=r^2$
and we introduce $\psi_{\mu}(x_3)=\sqrt \mu \psi(\mu x_3)$.
We claim that there exists $\mu_0>0$ such that
$\psi_\mu(x_3)$ satisfies all the conditions above
for every $\mu<\mu_0$.
Concerning the first condition notice that
$$\frac 12 \int |\partial_{x_3} \psi_\mu|^2 dx_3
- \frac 1{p+1}
(\int |\Psi_0|^{p+1} dx_1 dx_2) (\int |\psi_\mu|^{p+1} dx_3 )
$$$$
=\frac 12 \mu^2 \int |\partial_{x_3} \psi|^2 dx_3
- \frac 1{p+1} \mu^{\frac{p-1}2}
(\int |\Psi_0|^{p+1} dx_1 dx_2) (\int |\psi|^{p+1} dx_3 ).$$
In particular, since $1<p<5$, the first condition follows for every $\mu<<1$.
Concerning the second condition notice that
$$\|\Psi_0(x_1, x_2)\psi_\lambda(x_3)\|_{\dot H}^2
= \int |\partial_{x_3} \psi_\mu |^2 dx_3 + \lambda_0 \int |\psi_\mu |^2 dx_3
=\mu^2 \int |\partial_{x_3} \psi |^2 dx_3 +\lambda_0 r^2$$
and hence we conclude by choosing $\mu$ small enough
and $r_0=r_0(\chi)$ such that $2 \lambda_0 r_0^2<\chi^2$.
\end{proof}

\begin{lemme}\label{lem:novanishing}
Let $\chi>0$ and $r_0(\chi)>0$ 
be as in Lemma \ref{lem:comparison}.
Let $r<r_0$  and $(u_n)$ be a sequence such that
$$u_n\in S_r\cap B_{\chi}, \quad 
\lim_{n\rightarrow \infty} E(u_n)=\inf_{u\in S_r\cap B_{\chi}} E(u)=J_r^\chi,$$
then 
\begin{equation}\label{lowerbound}
\liminf_{n\rightarrow \infty} \int |u_n|^{p+1} dx>0.
\end{equation}
\end{lemme}

\begin{proof}
Assume by contradiction that
\eqref{lowerbound} is false
then we get:
$$J_r^\chi=\lim_{n\rightarrow \infty} 
\frac 12\int |\nabla_x u_n|^2 dx + \frac 12 \int (x_1^2+ x_2^2) |u_n|^2 dx
\geq r^2 \frac{\Lambda_0}2.
$$
Notice that it gives a contradiction with Lemma \ref{lem:comparison}.
\end{proof}

We can now prove the non-vanishing of the minimizing sequences up to translation w.r.t. to the direction $x_3$. This follows from the next lemma, the proof of which is classical  
(see for instance \cite{BV}
where a similar argument is used to prove the existence of vortex),
however we keep it in order to be self-contained.

\begin{lemme}\label{lem:novanish}
Assume that $\sup (\|u_n\|_{L^2} + \|u_n\|_{\dot H})<\infty$ 
and there exists $\varepsilon_0>0$ such that
$$\|u_n\|_{L^{p+1}}>\varepsilon_0, \quad \forall n$$
then
for a sequence $(z_n)\in \R$ we have
$$u_n(x_1, x_2, x_3 - z_n) \rightharpoonup\bar u \neq 0, \quad \hbox{ in } H.$$ 
\end{lemme}

\begin{proof}
By interpolation we get
\begin{equation}\label{eq:lowerbound}\|u_n\|_{L^{2+4/3}}>\eta_0>0,
\quad \forall n.
\end{equation}
Moreover we have
$$\|u\|_{L^{2+\frac 43}(T_k)}^{2+\frac 43}\leq C \|u\|_{L^2(T_k)}^{\frac 43} \|u\|_{H^1(T_k)}^2$$
where
$$T_k=\R^2 \times (k, k+1), \quad k\in \Z$$
and hence by taking a sum over $k\in \Z$ we get:
$$\|u\|_{L^{2+\frac 43}}^{2+\frac 43}\leq C (\sup_k \|u\|_{L^2(T_k)})^{\frac 43} \|u\|_{H^1}^2.$$
Hence (due to the lower bound \eqref{eq:lowerbound} 
and due to the boundedness
of $(u_n)$ in $H^1$) we get
$$\exists k_n \hbox{ s.t. } \inf \|u_n(x_1, x_2, x_3)
\|_{L^2(T_{k_n})}>\delta_0>0$$ and the sequence
$w_n(x)=u_n(x_1, x_2, x_3 - k_n)$ satisfies
$$\sup (\int_{T_1} |\nabla_x w_n|^2 +\int_{T_1} (x_1^2 + x_2^2) |w_n|^2 dx
+\int_{T_1} |w_n|^2 dx)
<\infty,$$
$$\|w_n\|_{L^2(T_1)}>\delta_0.$$
By a compactness argument (that comes from the confining potential
$(x_1^2+ x_2^2)$) we deduce that $(w_n)$ has a non-trivial weak limit in $L^2(T_1)$.
\end{proof}

\subsection{Avoiding Dichotomy for Minimizing Sequences}

From now on we assume $\chi>0$ fixed 
and let $r_0=r_0(\chi)>0$ be the associated number
via Lemma \ref{lem:localmin}.

Our next lemma will be crucial to obtain the compactness of minimizing sequences.
\begin{lemme}\label{lem:dic}
Let $1 < p<5$ be fixed.
We have the following inequality for $0<r<s<\min\{1, r_0\}$
$$ r^2 J_s^\chi < s^2 J_r^\chi.$$
\end{lemme}

\begin{proof}
Let $(v_n) \subset S_r\cap B_{\chi}$ be such that $\lim_{n\rightarrow \infty}
E(v_n)=J_r^\chi$.
Notice that, by Lemma \ref{lem:localmin} and since  $r<r_0(\chi)$ we can assume $v_n\in B_{\chi r}$ for every 
$n$ large enough.
In particular we have
$$\frac sr v_n\in S_s\cap B_{\chi s}\subset S_s\cap B_{\chi}$$
provided that $s<1$. In particular we get
$$J_s^\chi\leq 
E(\frac sr v_n)$$$$=\frac{s^2}{r^2}\big (\frac 12\int |\nabla_x v_n|^2 dx +
\frac 12 \int (x_1^2+ x_2^2) |v_n|^2 dx\big )
- \frac 1{p+1} \frac{s^{p+1}}{r^{p+1}} \int |v_n|^{p+1} dx.$$
Recall that by Lemma \ref{lem:novanishing} we can assume that
$$ \frac{1}{p+1}\|v_n\|_{L^{p+1}}^{p+1}>\delta_0>0$$
and hence we can continue the estimate above as follows
\begin{align*}
...&=
 \frac{s^2}{r^2}\big (\frac 12 \int |\nabla_x v_n|^2 dx 
+ \frac 12
\int (x_1^2+ x_2^2) |v_n|^2 dx
- \frac 1{p+1} \int |v_n|^{p+1} dx\big )
+\frac 1{p+1} 
\big (\frac{s^2}{r^2}\\
&\qquad- \frac{s^{p+1}}{r^{p+1}}\big)\int |v_n|^{p+1} dx \\
&\leq \frac{s^2}{r^2}E(v_n) + \big (\frac{s^2}{r^2} - \frac{s^{p+1}}{r^{p+1}}
\big ) \delta_0
\leq \frac{s^2}{r^2} J_r^\chi  +\big (\frac{s^2}{r^2}- \frac{s^{p+1}}{r^{p+1}}\big ) \delta_0 + o(1)
<J_r^\chi \frac{s^2}{r^2},
\end{align*}
for every $n$ large enough.
\end{proof}

\subsection{Conclusion of the Proof of Theorem \ref{th:main}}
It is now classical, see for example \cite{Ca}, that the orbital stability of the set ${\mathcal M}_r^\chi$ is equivalent to the fact that any minimizing sequence $(u_n) \subset S_r\cap B_\chi  $
is compact up to translation. Namely that there exists $ (k_n) \subset \R $
s.t.
$ u_n(x_1, x_2, x_3 - k_n)$  is compact in $H$.
Let us
fix $(u_n) \subset S_r\cap B_{\chi}$ such that $\lim_{n\rightarrow \infty}
E(u_n)=J_r^\chi$. By Lemma \ref{lem:novanish}
we get the existence
of $(k_n) \subset \R$ such that 
$w_n\rightharpoonup\bar w \neq 0$
where $(w_n)=(u_n(x_1, x_2, x_3 - k_n))$.
We claim that if $\|\bar w\|_{L^2}=r$ then the strong convergence of $(w_n)$ to $\bar w$ in $H$ holds. Indeed if $\|\bar w\|_{L^2}=r$ then $w_n \to \bar w$ strongly in $L^2(\R^3)$ and using the Gagliardo-Sobolev inequality \eqref{GN} we get that $w_n \to w$ strongly in $L^{p+1}(\R^3)$. We then deduce, from the weak convergence in $H$, that $J(w) \leq \lim J(w_n) = J_r^{\chi}$. Now if we assume that $\|w\|_{H}^2 < \liminf_{n\rightarrow \infty} \|w_n\|_{H}^2$ we obtain the contradiction that $J(w) < J_r^{\chi}$. Thus $\|w\|_H^2 = \lim_{n\rightarrow \infty} \|w_n\|_{H}^2$ and, using again the weak convergence, we deduce that $w_n \to w$ strongly in $H$. 
To prove that $\|\bar w\|_{L^2}=r$ we assume by contradiction 
that $\|\bar w\|_{L^2}=\bar r<r$ and split
$w_n=(w_n -\bar w)+ \bar w$. We have
$$\|w_n\|_{\dot H}^2= \|w_n -\bar w\|_{\dot H}^2+ \|\bar w\|_{\dot H}^2+ o(1)$$
$$\|w_n\|_{L^2}^2= \|w_n -\bar w\|_{L^2}^2+ \|\bar w\|_{L^2}^2+ o(1).$$
Moreover by using the Brezis-Lieb Lemma~\cite{BL} we also have
$$\|w_n\|_{L^{p+1}}^{p+1}= \|w_n -\bar w\|_{L^{p+1}}^{p+1}+ \|\bar w
\|_{L^{p+1}}^{p+1}+ o(1)
$$
and hence
$$E(u_n)= E(w_n - \bar w) + E(\bar w)+ o(1) \geq J_{r_n}^\chi + J_{\bar r}^\chi + o(1)$$
where $r_n^2=\|w_n -\bar w\|_{L^2}^2$ and $\bar r^2=\|\bar w\|_{L^2}^2$,
hence $r_n^2+ \bar r^2= r^2+ o(1)$.
We can also assume that $r_n\rightarrow l$
and hence
$l^2+ \bar r^2=r^2$ which implies,
by the identity above and by Lemma \ref{lem:dic}
$$J_r^\chi\geq J_{l}^\chi + J_{\bar r}^\chi> \frac{l^2}{r^2}J_r^\chi + \frac{\bar r^2}{r^2}J_r^\chi=J_r^\chi.$$ 
This contradiction proves that necessarily
$\bar w\in S_r$. \qed
\section{Proof of Theorem \ref{complex}}

This section is devoted to the proof of Theorem \ref{complex}.

\subsection{Characterization of \texorpdfstring{$\C$}{C}-valued Minimizers}
Let $w\in H^1(\R^3; \C)$ be a complex valued minimizer. A standard elliptic regularity bootstrap ensures that $w$ is of class $C^1$.
It is well-known by the diamagnetic inequality
that also $|w|\in C^1(\R^3; \R)$ is a minimizer. Moreover by the Euler-Lagrange equation
and by using the strong maximum principle we get
$|w|>0$ and thus $w\in C^1(\R^3; \C\setminus\{0\})$.
Now observe that since $w$ and $|w|$ are minimizers, and since
all the terms involved in the energy (that we are minimizing) are unchanged
by replacing $w$ by $|w|$ except in principle the kinetic term,
we deduce that the unique possibility for $w$ and $|w|$ to be both minimizers
is that 
$\int_{\R^3} |\nabla_{x} |w||^2 dx=\int_{\R^3} |\nabla_{x} w|^2 dx$.
We conclude by Theorem \ref{elem} given in the Appendix.   

\subsection{Symmetry of Minimizers}
We now focus on the symmetry of the minimizers.
As pointed out to us by A. Farina, moving planes techniques as in \cite{LiNi} could likely be used to obtain the radial symmetry and monotonicity properties w.r.t. the $(x_1, x_2)$ variables and $x_3$ variable respectively. We have chosen here however to proceed differently and in particular to rely on the Schwartz symmetrization and reflexion type arguments. Doing so we are lead to establish results, see in particular Theorem \ref{nab}, that we think have their own interest and could be useful in other contexts where in principle the moving plane techniques do not work (e.g. the case of systems.


Assume that
 $u(x_1,x_2,x_3)$ is a real  minimizer. We introduce the partial symmetrization
w.r.t to the variables $(x_1, x_2)$:
\[
\tilde u(x_1, x_2,x_3)=u_{x_3}^*(x_1, x_2)
\]
where $u_{x_3}(x_1, x_2)=u(x_1,x_2,x_3)$ and $*$ denotes the Schwartz
rearrangement w.r.t. $(x_1, x_2)$ (see~\cite{LiLo}).
The following properties hold
$$\int_{\R^3} |\nabla_{x} \tilde u|^2 dx\leq \int_{\R^3} 
|\nabla_{x} u|^2 dx, 
\int_{\R^3} |\tilde u|^2 dx = \int_{\R^3} 
|u|^2 dx $$
$$
\int_{\R^3} |\tilde u|^{p+1} dx = \int_{\R^3} 
|u|^{p+1} dx$$
see \cite[Theorem 8.2]{BrSo}
and by \eqref{cam1} (from the appendix) we get
\begin{equation}\label{comprt}\int_{\R^2} (x_1^2+ x_2^2) |\tilde u(x_1, x_2,x_3)|^2 dx_1dx_2
\leq \int_{\R^2} 
(x_1^2+x_2^2)|u(x_1, x_2)|^2 dx_1dx_2, \quad \forall x_3\end{equation}
that by integration w.r.t to $dx_3$ gives
$$\int_{\R^3} (x_1^2+ x_2^2) |\tilde u|^2 dx\leq \int_{\R^3} 
(x_1^2+x_2^2)|u|^2 dx.$$
As a consequence we deduce that $\tilde u$ is also a minimizer.
Moreover since $u$ is a minimizer, then necessarily 
$$\int_{\R^3} (x_1^2+ x_2^2) |\tilde u|^2 dx= \int_{\R^3} 
(x_1^2+x_2^2)|u|^2 dx.$$ By this fact 
and \eqref{comprt} we get 
$$\int_{\R^2} (x_1^2+ x_2^2) |\tilde u|^2 dx_1dx_2= \int_{\R^2} 
(x_1^2+x_2^2)|u|^2 dx_1dx_2, \hbox{ a.e. } x_3\in \R$$
and hence by Theorem \ref{nab} (with $V(x_1, x_2)=x_1^2+ x_2^2$)
 we get $u(x_1, x_2,x_3)=\tilde u(x_1, x_2,x_3)$ for $\hbox{ a.e. } x_3\in \R$.

Summarizing $u(x_1,x_2,x_3)$ is radially symmetric and decreasing w.r.t. $(x_1,x_2)$ for 
 $x_3$ in a set with full measure. On the other hand 
 $u(x_1,x_2,x_3)$ is continuous and hence
$u(x_1,x_2,x_3)$ is radially symmetric and decreasing w.r.t. $(x_1,x_2)$ 
for every $x_3$.

We now establish that $u(x_1,x_2,\cdot)$ is even, namely radial with respect to the $x_3$ variable, after a suitable translation in the $x_3$-direction. To obtain this result we follow the approach introduced by O. Lopes \cite{Lo}, see also \cite[Theorem C.3]{Wi}. Clearly, up to a suitable translation along $(0,0,1)$, we can assume that
$$\int_{\R^2 \times [0, + \infty[} |u(x)|^2 dx = \int_{\R^2 \times ]- \infty, 0]} |u(x)|^2 dx.$$
We then define 
$$E_+(u) = \int_{\R^2 \times [0, + \infty[} |\nabla_x u|^2 + (x_1^2 + x_2^2)|u|^2 dx - \int_{\R^2 \times [0, + \infty[} |u(x)|^{p+1} dx$$
$$E_-(u) = \int_{\R^2 \times ]-\infty, 0]} |\nabla_x u|^2 + (x_1^2 + x_2^2)|u|^2 dx - \int_{\R^2 \times ]- \infty, 0]} |u(x)|^{p+1} dx.$$
Let $v$ be the reflection with respect to $x_3=0$ of $u$ restricted to $\R^2 \times [0, + \infty[$. Namely $v(x_1,x_2,x_3) = u(x_1,x_2,x_3)$ if $x_3 \geq 0$ and $u(x_1,x_2,-x_3)$ if $x_3 \leq 0$. Using the regularity of $u$ we deduce that $v \in S_r \cap B_{\chi}$ and then
$$E_+(u) + E_-(u) \leq 2 E_+(u).$$
Similarly by considering $w(x_1,x_2,x_3)= u(x_1,x_2, - x_3)$ if $x_3 \geq 0$ and $u(x_1,x_2,x_3)$ if $x_3 \leq 0$ we obtain that
$$E_+(u) + E_-(u) \leq 2 E_-(u).$$

Hence $E_-(u) = E_+(u)$ and we deduce that $v$ is also a minimizer of $J_r^{\chi}.$ Then there exists $\lambda_1, \lambda_2 \in \R$ such that
\begin{align*}
- \Delta u + (x_1^2 + x_2^2) u - |u|^{p-1}u &= \lambda_1 u\\
- \Delta v + (x_1^2 + x_2^2) v - |v|^{p-1}v &= \lambda_2 v.
\end{align*}
Since $u=v$ on $\R^2 \times [0, + \infty[$ it follows that $\lambda_1 = \lambda_2 := \lambda$ and thus $z$ defined by $z=u-v$ satisfies the linear equation $- \Delta z = L(x) z$ where
$$L(x) = - (x_1^2 + x_2^2) + \lambda + a(x) \quad \mbox{with} \quad
a(x) = \int_0^1 p |v+t(u-v)|^{p-2}(v+ t(u-v)) dt.$$
At this point we conclude, from the Unique Continuation Principle \cite{Ke}, that $z=0$. Thus $u=v$ and $u$ is even.

To conclude we still need to prove that $u(x_1,x_2, \cdot)$ is decreasing as a function of $x_3$.  
As after a suitable translation $k$ in the $x_3$-direction $u$ is even in $x_3$. Without loss of generality, we assume $k=0$ and $u$ is even in $x_3$. Consider for this the partial symmetrization
w.r.t to the variables $x_3$:
\[
 u^\dag(x_1, x_2,x_3)=u_{x_1, x_2}^*(x_3)
\]
where $u_{x_1, x_2}(x_3)=u(x_1,x_2,x_3)$ and $*$ denotes the Schwartz
rearrangement w.r.t. $x_3$.
The following properties hold
$$\int_{\R^3} |\partial_{x_i} u^\dag |^2 dx\leq \int_{\R^3} 
|\partial_{x_i} u|^2 dx, \mbox{ for }i=1,2,\, 
\int_{\R^3} |u^\dag |^2 dx= \int_{\R^3} 
|u|^2 dx,
$$
$$\int_{\R^3} |u^\dag |^{p+1} dx= \int_{\R^3} 
|u|^{p+1} dx$$
see \cite[Corollary 8.1]{BrSo} and 
$$\int_{\R} |\partial_{x_3}u^\dag |^2 dx_3\leq \int_{\R} 
|\partial_{x_3}u|^2 dx_3 ,\,\int_{\R} |u^\dag |^2 dx_3= \int_{\R} 
|u|^2 dx_3 $$
for almost all $(x_1,x_2)$ in $\R^2$ multiplying by $x_1^2+x_2^2$ the last identity and integrating it over $\R^2$ we obtain
$$
\int_{\R^3} (x_1^2+ x_2^2) |u^\dag(x_1, x_2,x_3)|^2 dx_1dx_2
= \int_{\R^3} 
(x_1^2+x_2^2)|u(x_1, x_2,x_3)|^2 dx_1dx_2dx_3.$$

As a consequence we deduce that $u^\dag$ is also a minimizer and necessarily
$$\int_{\R} |\partial_{x_3} u^\dag |^2 dx_3= \int_{\R} 
|\partial_{x_3} u|^2 dx_3$$
for a.e. $(x_1,x_2)$ in $\R^2$. We now use \cite[Lemma 3.2]{BroZi} characterizing functions satisfying this equality. Some comments are in order. This lemma is stated for compactly supported functions but can be extended as long as one assumes, as in \cite[Lemma 9]{BJM-2}, that $(u-t)_+\in H^1(\R)$ and has compact support for any $t>0$, see also \cite[Remark 4.5]{BroZi}. This extension uses ideas similar to the proof of Theorem~\ref{nab} below. 
In the present case $u$ is $C^1$ and tends to $0$ at infinity in $x_3$. Indeed each of its partial derivative in $x_3$ is square integrable for almost all (hence all by continuity) $(x_1,x_2)\in\R^2$. So $(u-t)_+\in H^1(\R)$ and has compact support for any $t>0$. We thus have with \cite[Lemma 3.2]{BroZi} that the level sets $\{|u(x_1,x_2,\cdot)|>t\}$ are, up to a negligible set in $(x_1, x_2)$, intervals in $\R$. Since $u$ is continuous this is exactly an open interval. Then as $\{|u(x_1,x_2,\cdot)|\geq t\}=\cap_{s<t}\{|u(x_1,x_2,\cdot)|> s\}$ it is an interval which is closed.

As a consequence of this fact we deduce that for any given $(x_1, x_2)$
the functions $u_{x_1, x_2}:\R\ni x_3\rightarrow u(x_1,x_2, x_3)\in \R$
are decreasing for $x_3\geq k(x_1, x_2)$ and increasing for 
for $x_3\leq k(x_1, x_2)$, where $k(x_1, x_2)\in \R$ is any point where the function
$u_{x_1, x_2}$ has a maximum. We conclude provided that we show that 
$k(x_1, x_2)=0$. Notice that if by contradiction it is not true, then by the
evenness of the function $u_{x_1,x_2}$ (proved above) we get that $-k(x_1, x_2)\neq k(x_1, x_2)$ are both points of maximum.
Hence by convexity (proved above) of the set 
$\{x_3\in \R|u_{x_1, x_2}(x_3)\geq u_{x_1, x_2}(k(x_1, x_2))\}$
we deduce that every point belonging to the interval 
$[-k(x_1, x_2), k(x_1, x_2)]$ is  a maximum point for $u_{x_1, x_2}$, and in particular we can choose $k(x_1, x_2)=0$ as  maximum point of $u_{x_1, x_2}$. \qed

\begin{rem}
Finally we point out that our arguments can be adapted in the more general context described in Remark \ref{rem:extension}. Note that in this context the argument from \cite[Lemma 3.2]{BroZi} provides that all level 
sets in the unconfined variable (freezing the confined ones) are balls while the argument from \cite{Lo} provides that any hyperplan (in the unconfined subspace) is up to a translation (in the unconfined variables) a symmetry hyperplan. Coupling both arguments we have that all the level sets are concentric balls by considering any hyperplan orthogonal to lines joining the centres of any pair of such balls.\\ 
Note that without the the argument from \cite{Lo}, \cite[Lemma 3.2]{BroZi} gives that up to a translation in the unconfined variables (which may depend on the confined ones) the minimizers is decreasing along any ray from the origin. But it is not necessarily radially decreasing.
\end{rem}

\subsection{Lagrange Multipliers} Since $u\in \mathcal M_r^\chi$ 
implies $u\in B_{\chi r}$ (see Theorem \ref{th:main}),
we deduce that $u$ is a critical point of $E(u)$ on $S_r$ and hence 
there exists $\lambda=\lambda(u) \in \R$ such that
$$- \Delta u + (x_1^2+x_2^2)u - u |u|^{p-1} = \lambda u.$$
Multiplying by $u$ and integrating by parts we get
\begin{equation}\label{145}
\lambda = \frac{1}{\int |u|^2 dx}\big ( \|u\|_{\dot H}^2 - \int |u|^{p+1} dx\big )
<\frac{2 E(u)}{r^2}
\end{equation}
where we have used $p>1$. Moreover by Lemma \ref{lem:comparison} 
 $E(u) < r^2 \Lambda_0/2$, and hence 
$$\lambda < \frac{2 E(u)}{r^2} < \Lambda_0.$$
On the other hand since
\begin{equation}\label{eq:estimateLp} 
\int |u|^{p+1} dx \leq C r^{\frac{5-p}{2}} \|u\|_{\dot H}^{\frac{3p-3}{2}}
\end{equation}
by \eqref{GN}, we obtain from (\ref{145}) and the fact that 
$u\in B_{\chi r}$ (see Theorem \ref{th:main})
that

\begin{eqnarray}\label{1000}
\lambda
& \geq&  \frac{1}{r^2}\big ( \|u\|_{\dot H}^2 - C r^{\frac{5-p}{2}}\|u\|_{\dot H}^{\frac{3p-3}{2}}\big) \geq 
 \frac{\|u\|_{\dot H}^2}{r^2}\big( 1 - C r^{\frac{5-p}{2}}\|u\|_{\dot H}^{\frac{3p-7}{2}}
\big) \nonumber
\\
& \geq & \frac{\|u\|_{\dot H}^2}{r^2}\big(1 - C r^{\frac{5-p}{2}}(\chi^2 r^2)^{\frac{3p-7}{4}}\big) 
\geq  \frac{\|u\|_{\dot H}^2}{r^2}\big( 1 - C r^{p-1}\big) \nonumber
\end{eqnarray}
where the value of the constant $C>0$ can change at every step.
Now by definition of $\Lambda_0$ 
$$\frac{\|u\|_{\dot H}^2}{r^2} \geq \Lambda_0$$
and it follows that
$$\lambda \geq \Lambda_0 (1 - C r^{p-1}).$$
In summary, for some constant $C>0$,
$$(1 - C r^{p-1}) \Lambda_0 \leq \lambda < \Lambda_0$$
and in particular $\lambda \to \Lambda_0$ as $r \to 0$.\qed


\subsection{Asymptotic Profile in the Small Soliton Limit}
We are now in position to complete the proof of Theorem~\ref{complex} by establishing  the asymptotic profile in the limit $r\to 0$ (see \eqref{eq:profile}).
Recall that from Theorem~\ref{th:main}, for every fixed $\chi>0$, for every $r<r_0(\chi)$ , 
\[
{\mathcal M}_r^{\chi}:= \{u\in S_r \cap {B}_{\chi}
\mbox{ s.t. } E(u)=J_r^\chi\}\neq \varnothing.
\]
First observe that since 
$${\mathcal M}_r^{\chi}\subset B_{\chi r}, \quad \forall r<r_0$$
we have, using  \eqref{eq:estimateLp},  that 
\begin{equation}\label{44}
\exists C>0,\,\forall u\in  {\mathcal M}_r^{\chi},\, \|u\|_{L^{p+1}}^{p+1}\leq C r^{p+1}.
\end{equation}

Now using again the Hilbert basis $(\Psi_j)$ of $L^2(\R^2)$  introduced in the proof of Lemma \ref{lem:spect} we write for any $u\in {\mathcal M}_r^{\chi}$:
\[
 u(x_1,x_2,x_3)=\sum_{j\in \N \cup \{0\}} \varphi_{j}(x_3)\Psi_j(x_1,x_2),
\]
where 
\[
\varphi_j(x_3)=\int u(x_1,x_2,x_3) \Psi_j(x_1,x_2) dx_1dx_2. 
\]
We have $\varphi_{j} \in L^2(\R)$ and by orthogonality
\[
  \sum_{j\in\N\cup\{0\}} \|\varphi_{j}\|_{L^2}^2=\sum_{j\in\N\cup\{0\}} \|\varphi_{j}\|_{L^2}^2\|\Psi_j\|_{L^2}^2=\|u\|_{L^2}^2=r^2.
\]
Also, as in the proof of Lemma \ref{lem:spect}, and taking into account (\ref{44})
\begin{equation}\label{45}
E(u)\geq  \frac{1}{2}\sum_{j\in\N\cup\{0\}}  \|\partial_{x_3} \varphi_{j}\|_{L^2}^2 + \lambda_j \|\varphi_{j}\|_{L^2}^2 - Cr^{p+1}.
\end{equation}
Now we know from Lemma \ref{lem:comparison} 
and Lemma \ref{lem:spect} that $$ E(u) = \inf_{S_r \cap B_{\chi}} E(u)<r^2\frac{\lambda_0}2, \quad \forall r<r_0$$ and thus (\ref{45}) leads to 
\begin{equation}\label{eq:energyestimate}
\frac1{r^2} \sum_{j\in\N\cup\{0\}} \|\partial_{x_3} \varphi_{j}\|_{L^2}^2+ \lambda_j \|\varphi_{j}\|_{L^2}^2 < \lambda_0+ 2Cr^{p-1}.
\end{equation}
In view of the following identity 
\[
\frac{\lambda_0}{r^2} \sum_{j\in\N\cup\{0\}} \|\varphi_{j}\|_{L^2}^2=\frac{\lambda_0}{r^2} \|u\|_{L^2}^2=\lambda_0
\]
and by recalling $\lambda_0< \lambda_j$ for all $j \in \N$,
we get from the estimate above
\begin{equation}\label{cigar}
\sum_{j\in\N\cup\{0\}} \frac1{r^2}\|\partial_{x_3} \varphi_{j}\|_{L^2}^2<  2Cr^{p-1}
\quad\mbox{ and }\quad
\sum_{j\in\N} \frac1{r^2}\|\varphi_{j}\|_{L^2}^2<  \frac{2C}{\inf_{j\in\N}\left(\lambda_j-\lambda_0\right)}r^{p-1}
\end{equation}
or incidentally, 
\begin{equation}\label{eq:coordinateestimate} 
\left|\frac{\|\varphi_{0}\|_{L^2}^2}{r^2}- 1\right|<  \frac{2C}{\inf_{j\in\N}\left(\lambda_j-\lambda_0\right)}r^{p-1}.
\end{equation}
From \eqref{eq:energyestimate}, \eqref{cigar} and \eqref{eq:coordinateestimate} it follows that there exists $C>0$ such that for any $u\in {\mathcal M}_r^{\chi}$
\[
 \left\|\frac{u}{r} -\frac{\varphi_{0}}{r} \Psi_0\right\|_{\dot{H}} =  \sqrt{\frac1{r^2}\sum_{j\in\N}\|\partial_{x_3}\varphi_{j}\|_{L^2}^2+\lambda_j\|\varphi_{j}\|_{L^2}^2 }\leq C r^{\frac{p-1}{2}},
\]
and we get \eqref{eq:profile}.


\qed



%
%
%
%
\section{Proof of Theorem \ref{ground-state}}

We fix $\chi>0$ and  
notice that
\begin{equation}\label{eq:asyJ}
\lim_{r \rightarrow 0} J_{r}^{\chi} =0.
\end{equation}
This follows from Point \eqref{th:mainPoint2} in Theorem~\ref{th:main} and \eqref{GN}.

Now let us suppose by contradiction the existence of a critical point $\bar u$ for $E$ on $S_r$ with $E(\bar u)<J_r^{\chi}$. It is standard to prove, see for example \cite{BeJe,Je}, that 
$P(\bar u)=0$ where 
$$P(\bar u):=\|\nabla_x \bar u\|_{L^2}^2- \int (x_1^2+ x_2^2) |\bar u|^2 dx-\frac{3(p-1)}{2(p+1)}||\bar u||_{L^{p+1}}^{p+1}.$$ 
So to say, the condition $P(\bar u)=0$ corresponds to a Pohozaev's type identity on $S_r$. 
Thus we can write
$$
E(\bar u)=E(\bar u)-\frac{2}{3(p-1)}P(\bar u)=\frac{3p-7}{6(p-1)}\|\nabla \bar u\|_{L^2}^2+\frac{3p+1}{6(p-1)} \int (x_1^2+ x_2^2) |\bar u|^2 dx>\frac{3p-7}{6(p-1)}\|\bar u\|_{\dot H}^2.
$$
In particular, by using \eqref{eq:asyJ}, we get
\[ 
\frac{3p-7}{6(p-1)}\|\bar u\|_{\dot H}^2 < E(\bar u )<J_r^{\chi}=o(1).
\]
At this point we have reached a contradiction since  this implies that, for $r>0$ small enough, $\|\bar u\|_{\dot H}^2<\chi^2$ and $J_r^{\chi}$ is the infimum of the energy on $S_r \cap B_{\chi}$. \qed

\appendix

\section{A remark on Schwartz symmetrization}

The main result that we prove on Schwartz symmetrization is the following one. \medskip

For the definition of a function vanishing at infinity, used in the next theorem, we refer to \cite{LiLo}. 
\begin{theoreme}\label{nab}
Let $V:\R^n\rightarrow [0, \infty)$ be a measurable function,
radially symmetric satisfying $V(|x|)\leq V(|y|)$ for $|x|\leq |y|$ then we have:
\begin{equation}\label{cam1}
\int_{\R^n} V(|x|)|u^*|^2 dx\leq \int_{\R^n}  V(|x|) |u|^2 dx.
\end{equation}
If in addition $V(|x|)<V(|y|)$ for $|x|<|y|$ then 
\begin{equation}\label{cam2}
\int_{\R^n}  V(|x|)|u^*|^2 dx= \int_{\R^n}  V(|x|) |u|^2 dx \Rightarrow u(x)=u^*(|x|).
\end{equation}
This result holds for any measurable function $u$ which is vanishing at infinity.
\end{theoreme}
In the sequel we shall use the following well-known inequality
\begin{equation}\label{fstar}
\int_{\R^n} f^* \cdot g^* dx\geq \int_{\R^n} f \cdot g dx.
\end{equation}

\begin{proposition}\label{LL}
Let $u: \{x\in \R^n||x|<R\}\rightarrow [0, \infty)$ be a measurable function
and let $V:\{x\in \R^n||x|<R\}\rightarrow [0, \infty)$ be radially symmetric
satisfying $V(|x|)>V(|y|)$ for $|x|<|y|<R$.
Then we have the following implication
$$\int_{|x|<R} V(|x|) |u^*|^2 dx= \int_{|x|<R} V(|x|) |u|^2dx
\Rightarrow u(x)=u^*(|x|).$$
\end{proposition}

\begin{proof}
This result is classical in the case $R=\infty$, more precisely
for functions defined on $\R^d$, see \cite[Theorem 3.4]{LiLo}. The proof given there directly adapt on any ball $\{x\in \R^n||x|<R\}$.
\end{proof}

\begin{proposition}\label{boh}
Let $V:\R^n\rightarrow [0, \infty)$ be a measurable function,
radially symmetric with $V(|x|)>V(|y|)$ for $|x|<|y|<R$ and $V(x)=0$ for $|x|>R$.
Let $u:\R^n\rightarrow [0, \infty)$ be measurable and vanishing at infinity.
The following implication holds:
$$\int_{\R^n} V(|x|) |u^*|^2= \int_{\R^n} V(|x|) |u|^2 dx
\Rightarrow u(x)\cdot \chi_{|x|<R} \hbox{ is radially symmetric and decreasing. }$$
\end{proposition}

\begin{proof}
We defined $u_R(x)= u(x) \cdot \chi_{|x|<R}(x)$.
First note that by \eqref{fstar} we have 
\begin{equation}\label{1}
\int_{\R^n} V(|x|) |u_R^*|^2\geq  \int_{\R^n} V(|x|) |u_R|^2
\end{equation}
where we used that  $V(x)=V^*(x)$ holds since $V$ is radially symmetric decreasing, see \cite[Proposition 2.4]{HaSt1}. Now by elementary considerations 
since $u_R(x)\leq u(x)$
we have $u_R^*(x)\leq u^*(x)$
and hence 
$$\int_{\R^n} V(|x|) |u^*|^2 \geq  \int_{\R^n} V(|x|) |u_R^*|^2\geq
\int_{\R^n} V(|x|) |u_R|^2=\int_{\R^n} V(|x|) |u|^2$$
where
in the second inequality we have used \eqref{1} and in the last identity
the assumption $\supp V(x)=\{x\in \R^n||x|<R\}$.
Since by assumption 
$$\int_{\R^n} V(|x|) |u^*|^2= \int_{\R^n} V(|x|) |u|^2 dx$$ 
we get 
by the inequalities above $\int_{\R^n} V(|x|) |u_R^*|^2= \int_{\R^n} V(|x|) |u|^2 dx$, 
that in turn, by using again the assumption $\supp V(x)=B_R$, implies
$$\int_{B_R} V(|x|) |u_R^*|^2=\int_{\R^n} V(|x|) |u_R^*|^2=
\int_{\R^n} V(|x|) |u|^2=\int_{B_R} V(|x|) |u_R|^2.$$
We conclude by Proposition \ref{LL}.
\end{proof}

{\bf Proof of Theorem~\ref{nab}}
To prove \eqref{cam1} notice that for every $R>0$ we can introduce the function
$V_R(|x|)=\min \{V(|x|),V(R)\}$.
Then we get that
$x\rightarrow -V_R(|x|) +V(R)$ is decreasing, radially symmetric
and positive.
Hence
by \eqref{fstar} we get
$$\int_{\R^n}  (-V_R(|x|) +V(R)) |u|^2 dx\leq  \int_{\R^n}  (-V_R(|x|) +V(R)) |u^*|^2 dx.$$
Since moreover $\int_{\R^n}  |u|^2 dx=\int_{\R^n}  |u^*|^2 dx$,
then we deduce
$$\int_{\R^n}  (V_R(|x|) |u|^2 dx\geq \int_{\R^n}  (V_R(|x|) |u^*|^2 dx.$$
We conclude the proof of \eqref{cam1}
since $\lim_{R\rightarrow \infty}
V_R(|x|)=V(|x|)$ and by using the Beppo Levi monotonic
convergence theorem.

Next we prove \eqref{cam2}. We claim that
\begin{equation}\label{claimbes}\int_{\R^n}  V_{R_0}
(|x|)|u^*|^2 dx= \int_{\R^n}  V_{R_0}(|x|) |u|^2 dx
\quad \forall R_0>0\end{equation}
where $V_{R_0}(|x|)$ is defined as above with a truncation procedure.
Notice that \eqref{claimbes} is equivalent to
$$
\int_{\R^n}  (-V_{R_0}(|x|) + V(R_0))|u^*|^2 dx = \int_{\R^n}  (-V_{R_0}(|x|) + V(R_0)) |u|^2 dx.
$$
Since the potential $x \to -V_{R_0}(|x|) + V(R_0)$ satisfies the assumptions
of Proposition \ref{boh} for $R=R_0$  we deduce that 
the restriction of $u$ on $\{x\in \R^n||x|<R_0\}$ is decreasing radially symmetric.
Since $R_0>0$ is arbitrary we conclude that $u(x)=u(|x|)$
and $u$ is decreasing on $\R^n$, and thus that $u = u^*$. 
In order to prove \eqref{claimbes} first 
notice that since $V-V_{R_0}$ is increasing and positive, then
we can apply \eqref{cam1} to deduce
\begin{equation}\label{inc}
\int_{\R^n}  (V(|x|)-V_{R_0}(|x|))|u^*|^2 dx\leq \int_{\R^n}  (V(|x|)-V_{R_0}(|x|)) |u|^2 dx.
\end{equation}
Assume now by contradiction that 
\eqref{claimbes} is false,
then we have 
\begin{equation}\label{inc2}\int_{\R^n}  V_{R_0} (|x|)|u^*|^2 dx< \int_{\R^n}  V_{R_0}(|x|) |u|^2 dx.\end{equation}
By summing \eqref{inc} and \eqref{inc2} we get
$\int_{\R^n}  V(|x|)|u^*|^2 dx< \int_{\R^n}  V(|x|) |u|^2 dx$,
which gives a contradiction with the hypothesis done in \eqref{cam2}. \qed


\section{On the canonical form of complex minimizers}

The result below is the key to derive the structure of the set of our complex minimizers.

\begin{theoreme}\label{elem}
Let $w\in C^1(\R^n;\C\setminus\{0\})$ be such that
$$\int_{\R^n} \sum_j |\partial_{x_j} |w||^2 dx=\int_{\R^n} \sum_{j} |\partial_{x_j} w|^2 dx
$$
then we have 
$$w(x)=e^{i\theta} \rho(x)$$ where $\theta \in \R$ is a constant and $\rho(x)\in \R$ for every $x\in \R^n$.
\end{theoreme}

\begin{proof}
We write $w = \rho u$ where $\rho=|w| >0$. Since $w \in C^1(\R^n, \C\setminus\{0\})$ it follows that $|w|$ and 
$u \in C^1(\R^n; \C\setminus\{0\})$. Using that $|u|=1$ we get
$$\nabla w = (\nabla \rho) u + \rho  \nabla u=u(\nabla \rho  + \rho \bar{u} \nabla u).$$
Again from $|u|=1$ it follows that $\Re(\bar{u}\nabla u)=0$ that is $\bar{u}\nabla u$ is purely imaginary. Then $|\nabla w|^2 = |\nabla \rho|^2 + \rho^2 |\nabla u|^2$. Now since $|w|= \rho $, we have that $|\nabla |w|| = |\nabla \rho|$.
 At this point our assumption gives
$$\int_{\R^n} |\nabla \rho|^2 dx= \int_{\R^n}( |\nabla \rho|^2  + |\rho \nabla u|^2) dx $$
and thus $\int_{\R^n} |\rho \nabla u|^2 dx = 0$ which leads to $\nabla u=0$.
\end{proof}

\bibliographystyle{plain}

\end{document}